\documentclass[12pt]{amsart}
\usepackage{pdfsync}
\usepackage{graphicx,color}
\usepackage{amssymb}
\usepackage{amsmath}
\usepackage{amsthm}
\usepackage{enumerate}
\usepackage{subcaption}
\usepackage{esvect}
\usepackage{ucs}
\usepackage{cite}
\usepackage{amsxtra}
\usepackage{epstopdf}
\usepackage{bbold}
\usepackage[utf8x]{inputenc}
\usepackage{verbatim}                                  %   \begin{comment}  \end{comment}
\usepackage[toc,page]{appendix}                        %  \appendix
\usepackage{pdfsync} % pdf inverse search

\usepackage[colorinlistoftodos]{todonotes}

\usepackage[pagebackref]{hyperref}

\newcommand{\dist}{\text{dist}}

\DeclareMathOperator{\diam}{diam}

\newtheorem{theorem}{Theorem}
\newtheorem{definition}[theorem]{Definition}

\newtheorem{corollary}[theorem]{Corollary}

\newtheorem{example}[theorem]{Example}

\newtheorem{remark}[theorem]{Remark}

\newcommand{\R}{\mathbb{R}}
\def\Z{{\mathbb Z}}
\def\N{{\mathbb N}}

\newcommand{\B}[1]{\mathbb{#1}}

\makeatletter
\def\moverlay{\mathpalette\mov@rlay}
\def\mov@rlay#1#2{\leavevmode\vtop{%
   \baselineskip\z@skip \lineskiplimit-\maxdimen
   \ialign{\hfil$\m@th#1##$\hfil\cr#2\crcr}}}
\newcommand{\charfusion}[3][\mathord]{
    #1{\ifx#1\mathop\vphantom{#2}\fi
        \mathpalette\mov@rlay{#2\cr#3}
      }
    \ifx#1\mathop\expandafter\displaylimits\fi}
\makeatother

\newcommand{\bigcupdot}{\charfusion[\mathop]{\bigcup}{\cdot}}

\begin{document}
\thispagestyle{empty}

\title{Numbers omitting digits in certain base expansions}

\author{Alexia Yavicoli}
\address{Department of Mathematics, The University of British Columbia. 1984 Mathematics Road, Vancouver BC V6T 1Z2, Canada}
\email{yavicoli@math.ubc.ca}

\author{Han Yu}
\address{Mathematics Institute,Zeeman Building,University of Warwick, Coventry CV47AL, United Kingdom}
\email{Han.Yu.2@warwick.ac.uk}


\makeatletter
\providecommand\@dotsep{5}
\makeatother

\begin{abstract}
In \cite{Yu} the author proved that for each integer $k$ there is an implicit number $M > 0$ such that if $b_1, \cdots , b_k$ are multiplicatively independent integers greater than $M$, there are infinitely many integers whose base $b_1, b_2, \cdots , b_k$ expansions all do not have zero digits.
In this paper we don't require the multiplicative independence condition and make the result quantitative, getting an explicit value for $M$. We also obtain a result for the case when the missing digit(s) may not be
zero. Finally, we extend our method to study various missing-digit sets in an algebraic setting.
\end{abstract}

\maketitle

\section{Introduction and main results}
In this paper, we utilise the notion of thickness for sets in $\mathbb{R}^d,d\geq 1$ to prove several results concerning numbers with missing digits in various different possibly algebraic bases. To start, we first consider the easy setting with `usual' missing digit sets with rational integer bases. 
\begin{definition}
We say that a natural number $N$ can be written in basis $b$ without the digit $0$, if there is a $j \in \N_0$ so that
\[N=\sum_{0\leq i\leq j} c_i b^i \text{ where } c_i \in \{1, \cdots, b-1\}.\]
We say that a positive real number $R$ can be written in basis $b$ without the digit $0$, if there is a $j \in \Z$ so that
\[R=\sum_{-\infty < i\leq j} c_i b^i \text{ where } c_i \in \{1, \cdots, b-1\}.\]
\end{definition}

In \cite{Yu} the author proved that for each integer $k$ there is an implicit number $M > 0$ such that if $b_1, \cdots , b_k$ are multiplicatively independent integers greater than $M$, there are infinitely many integers whose base $b_1, b_2, \cdots , b_k$ expansions all do not have zero digits.
In this paper,  we make the result quantitative, getting an explicit value for $M$, where the multiplicative independence hypothesis is not required.

\begin{theorem}\label{thm: main1}
Let $k\geq 2$ be an integer. Let $(b_1, \cdots, b_k)$ be a $k$-tuple of natural numbers. Let $M=M(k)\in \N_{\geq 5}$ be the smallest number satisfying that $g_k(M-2):=\frac{M-2}{\log(M-2)}-k \frac{e 4 (432)^2}{\log (4)}>0$. If $b_i \geq M=M(k)$ for every $1\leq i \leq k$, then there are infinitely many integers whose base $b_1, \cdots, b_k$ expansions all omit the digit $0$.
\end{theorem}
\begin{remark}
    For a concrete example, we can take $k=3.$ In this case, we need $M$ such that
    \[
    \frac{M-2}{\log(M-2)}>\frac{2239488e}{\log(4)}>4391245.
    \]
    Thus $M=79904626$ would be enough. It is convenient to find some easier ``gauge" functions to replace $g_k$. For example, for any $\tilde\varepsilon \in (0,1)$, if $k$ is large enough in terms of $\tilde\varepsilon$ then $M\approx k \log^{1+\tilde\varepsilon}(k)$ works.
\end{remark}

We also prove a result for when the missing digit(s) may not be zero. 
\begin{theorem}\label{thm: main2}
Let $k\geq 2$ be an integer, $(b_1, \cdots, b_k)$ be a $k$-tuple of natural numbers and $r\in [\frac{1}{\min_{1\leq i \leq k} b_i},1)$. Let $M_1, \cdots , M_k$ be sets (of missing digits) satisfying:
\begin{itemize}
\item $0 \in M_i \subseteq \{0, \cdots , b_i-2\}$,
\item $M_i$ does not have consecutive numbers,
\item $\{0, \cdots , b_i-1\}\setminus M_i$ is formed by rows of at least $b_i r$ consecutive numbers. 
\end{itemize}
 Let $M=M(k)\in \N_{\geq 4}$ be the smallest number satisfying that \[g_k(r(M-2)):=\frac{r(M-2)}{\log(r(M-2))}-k \frac{e 4 (432)^2}{\log (4)}>0.\] If $b_i \geq M=M(k)$ for every $1\leq i \leq k$, then there are infinitely many integers whose base $b_1, \cdots, b_k$ expansions omit the digits in the sets $M_1, \cdots, M_k$ respectively.

\end{theorem}

Finally, we extend our method to study various missing-digit sets in an algebraic setting. The sets in this scenario are self-similar, possibly with rotations. We defer the detailed description of all the terminologies and notations in Section \ref{sec: algebraic}, where we prove the following result.

\begin{theorem}\label{thm: algebraic}
     Let $\B{K}$ be a number field so that $O_{\B{K}}$ is an Euclidean domain. For each $k\geq 2,$ there is a number $M>0$ such that for each $k$ different integers in $O_{\B{K}}$, say, $b_1,\dots,b_k$ with equal norm Galois embeddings, and with norm at least $M$, there are infinitely many integers $z\in O_{\B{K}}$ with the property that $z$ misses one arbitrarily fixed digit in base $b_j$ for all $j$. 
\end{theorem} 
\begin{remark}
    Here $O_{\B{K}}$ is the ring of integers in $\B{K}$. We require that it is an Euclidean domain. This is more for convenience than necessity. Our method would work in general number fields with careful modifications. In this paper, we do not pursue this level of generality. Another condition is that $b_1,\dots,b_k$ are integers with equal norm Galois embeddings, e.g. integers in complex quadratic fields. This condition is crucial for the application of the thickness argument. Should this condition fail, the geometry of the missing-digit set would be not self-similar but self-affine with small thickness. 
\end{remark}
\begin{remark}
    This result can be made quantitative with the explicit constants provided in Theorem \ref{thm: FY} as well as the detailed knowledge of the number field $\B{K}$. We do not pursue this refinement in this paper.
\end{remark}

We need the notion of thickness on the real line. Recall that every compact set $C$ on the real line can be constructed by starting with a closed interval $I$ (its convex hull), and successively removing disjoint open complementary intervals (they are the path-connected components of the complement of $C$). Clearly there are finitely or countably  many disjoint open complementary intervals $(G_n)_n$, which we may assume are ordered so that their length $|G_n|$ is decreasing (if some intervals have the same length, we order them arbitrarily). We emphasize that the two unbounded connected components of $\mathbb{R}\setminus C$ are not considered as we start with the convex hull $I$. Note that $G_{n+1}$ is a subset of some connected component $I_{n+1}$ (a closed interval) of $I\setminus (G_1\cup\cdots\cup G_n)$. We say that $G_{n+1}$ is removed from $I_{n+1}$.

\begin{definition}
Let $C \subset \R$ a compact set with convex hull $I$, and let $(G_n)$ be the open intervals making up $I\setminus C$, ordered in decreasing length. Each $G_n$ is removed from a closed interval $I_n$, leaving behind two closed intervals $L_n$ and $R_n$; the left and right pieces of $I_n \setminus G_n$.

We define the thickness of $C$ as
\[
\tau (C):= \inf_{n \in \N} \frac{\min \{ |L_n| , |R_n| \}}{|G_n|}.
\]
Note that the sequence of complementary intervals $(G_n)_n$ may be finite, and in this case the infimum is taken over the finite set of indices.

We define the thickness of a singleton as $0$, and the thickness of a (non-degenerate) interval as $+\infty$.
\end{definition}

It can be checked that if there are some intervals of equal length, then the way in which we order them does not affect the value of $\tau(C)$. See \cite{Yav,HKY,Astels} for more information on Newhouse thickness and alternative definitions.

\begin{example}\label{excantor}
Let $M_{\varepsilon}$ be the middle-$\varepsilon$ Cantor set obtained by starting with the interval $[0,1]$ and repeatedly deleting from each interval appearing in the construction the middle open interval of relative length $\varepsilon$. Then $\tau (M_{\varepsilon})= \frac{1-\varepsilon}{2 \varepsilon}$.
\end{example}

For the algebraic setting, we need the following result from \cite[Theorem 6]{FalYav} to guarantee that intersections of thick sets in $\mathbb{R}^d,d\geq 1$ are non-empty. The definition of thickness can be generalised to $\mathbb{R}^d$. See \cite{FalYav}. 

Given a compact subset $C$ of $\mathbb{R}^d$, we define $(G_n)_{n=1}^\infty$ to be the (at most) countably many open bounded path-connected components of the complement $C$ and $E$ to be the unbounded open path-connected component of the complement $C$ (except when $d = 1$ when $E$ consists of two unbounded intervals). We call $E$ together with $G_n,n\geq 1$ the gaps of $C$. We may assume that the sequence of bounded gaps $G_n,n\geq 1$ is ordered by non-increasing diameter. We write $\mathrm{dist}$ for the usual distance between points or non-empty subsets of $\mathbb{R}^d$ and $\mathrm{diam}$ for the diameter of a non-empty subset.
\begin{definition}
    Let $C\subset\mathbb{R}^d$ be given as above. The thickness $\tau(C)$ is given as
    \[
    \tau(C)=\inf_{n\geq 1} \frac{\mathrm{dist}(G_n,\bigcup_{i\leq n-1}G_i\cup E)}{\mathrm{diam}(G_n)}
    \]
    if $E$ is not the only path-connected component of the complement of $C$. Otherwise, we define
    \[
    \tau(C)=\begin{cases}
        \infty \text{ if } C^\circ\neq\emptyset\\
        0 \text{ if } C^\circ=\emptyset
    \end{cases}
    \]
\end{definition}

\begin{theorem}\label{thm: FY}
Let $(C_i)_i$ be a family of countably many compact sets in $\R^d$, where $C_i$ has thickness $\tau_i>0$ such that
\begin{enumerate}
\item $\sup_i \diam (C_i)<\infty$,
\item there is a ball $B$ contained in the convex hull of every $C_i$,
\item there exists $c\in (0,d)$ so that \[\sum_i \tau_i^{-c}\leq \frac{1}{K_2}\beta^c (1-\beta^{1-c}),\]
where $\beta:=\min\{\frac{1}{4}, \frac{\diam (B)}{\sup_i \diam (C_i)} \}$ and $K_2:=\left(\frac{(24\sqrt{d})^d (1+4^d 2)}{1-\frac{1}{2^d}}\right)^2$.
\end{enumerate}
Then, \[\dim_H(B\cap \bigcap_iC_i) >0.\]
\end{theorem}
\begin{remark}
This result is quantitative. For example, if $d=1$, we have $K_2=432^2$.
\end{remark}
Note that this definition is not useful for totally disconnected sets in $\mathbb{R}^d$, but there is an alternative definition of thickness and related results in that context, see \cite{Yav2}.

To prove Theorem \ref{thm: algebraic}, we will apply this result for self-similar sets with uniform scaling ratio in $\mathbb{R}^d$. Let $K$ be a self-similar set in $\mathbb{R}^d$. Let $H$ be the convex hull of $K$. The self-similar set $K$ will have the property that $H$ is close (in the sense of Hausdorff metric) to be a cube of scale $\approx 1$. Let $f_1,\dots,f_k$ be a collection of similarity maps that defines $K$. Then $\bigcup_{j=1}^k f_i(H)$ is also close (in the sense of Hausdorff metric) to be a cube of scale $\approx 1$. Moreover, the complement of $K$ will not be path-connected. Such a self-similar set $K$ will have a large thickness. Next, given a multiple collection of such $K$'s, we will consider the case when their convex hulls are close in the sense of Hausdorff metric. We will achieve a scenario in which Theorem \ref{thm: FY} applies because the $\tau_i$'s are large enough numbers while $\beta$ can be bounded away from zero in a uniform manner.

\section{Omitting the digit $0$}

Let $b \in \N_{\geq 3}$, we define the compact set
\[\mathcal{C}_{b,0}:=\{x \in [0,1]: \ x=\sum_{i<0}c_i b^i \text{ with } c_i \in \{1, \cdots, b-1\}\}.\]
This is the set of numbers in the interval $[0,1]$ that can be written without the digit $0$ in base $b$.

One can see that the convex hull of $\mathcal{C}_{b,0}$ is $[\frac{1}{b-1},1]$, and its thickness is $b-2$. To see this, notice that because of the self-similarity and avoidance of an endpoint digit, we just need to look at the quotients involved in the definition of thickness for the first step of construction: from an interval of length $1-\frac{1}{b-1}$ we remove gaps of length $\frac{1}{b} \frac{1}{b-1}$ leaving intervals of length $\frac{1}{b}(1-\frac{1}{b-1})$, so the infimum of the quotients of the first level is $\frac{\frac{1}{b}(1-\frac{1}{b-1})}{\frac{1}{b} \frac{1}{b-1}}=b-2$.

We also define for every $j\in \N$ its expanded version by $b^j$ as \[\mathcal{C}_{b,j}:=b^j \mathcal{C}_{b,0}=\{x \in [0,b^j]: \ x=\sum_{i<j}c_i b^i \text{ with } c_i \in \{1, \cdots, b-1\}\}.\]
Its convex hull is $b^j [\frac{1}{b-1},1]$, its thickness is $b-2$, the length of the largest gaps is $b^j \frac{1}{b}\frac{1}{(b-1)}$ and the length of the intervals after removing the largest gaps is $b^j\frac{1}{b}(1-\frac{1}{b-1})$. Note that for a fixed $b$, the sets $\mathcal{C}_{b,j}$ are disjoint and moving to the right (when $j$ moves to infinity). Then, the set of positive numbers that can be written in base $b$ without the digit $0$ is
\begin{align*}A_{b} & :=\{x \in (0, \infty): \ x=\sum_{i<j}c_i b^i \text{ with } c_i \in \{1, \cdots, b-1\} \text{ for some } j \in \N\}\\
&= \bigcupdot_{j \in \N} \mathcal{C}_{b,j}.\end{align*}
\begin{proof}[Proof of Theorem \ref{thm: main1}]
We can assume that $b_1< \cdots < b_k$.
For each $i \in \{1, \cdots, k\}$ we define
\begin{align*}A_{b_i}&:=\{x\in (0, \infty): \ \exists N \in \Z \text{ so that } x\\ &=\sum_{j \in \Z_{\leq N}} c_j b_i^j \text{ with } c_j\in \{1,\cdots, b_i-1\}\},\end{align*}
the set of positive numbers for which the $b_i$-ary expansion does not contain the digit $0$.
And \[A:=\bigcap_{1\leq i \leq k}A_{b_i},\] the set of positive numbers that can be written in all bases $b_1, \cdots, b_k$ without the digit $0$.

It is enough to prove that $A$ is unbounded. Then we can find an arbitrarily large $x \in A$, so the integer part $[x] \in \Z$ is arbitrarily large and can be written omitting the digit $0$ in all bases $b_1, \cdots, b_k$ (when $ x=\sum_{j \in \Z_{\leq N}} c_j b_i^j$, the integer $[x]$ would be written as $\sum_{0\leq j \leq N} c_j b_i^j$).

The idea is that in order to prove that $A$ is unbounded, we will get a sequence of  $k$-tuples of compact subsets $(B_{1,n}, \cdots, B_{k,n})_n$, where $B_{i,n}\subseteq A_{b_i}$, the sets $B_{i,n}$ are thick and the convex hulls of sets in a tuple are not so different (in this way we will ensure a non-empty intersection, i.e: $\exists \, x_n \in \bigcap_{1\leq i\leq k} B_{i,n} \subseteq A$), but the sequence of tuples of sets that are alike will be moving to infinity.

We define a sequence of $k$-tuples $(a_j)_{j \in \N}$ where $a_j:=(a_{1,j}, \cdots, a_{k,j})$ as follows:
\begin{align*}
a_j&:=(b_1^j, b_2^{[j \log_{b_2}(b_1)]}, \cdots, b_k^{[j \log_{b_k}(b_1)]})\\
&=b_1^j(1, b_2^{-\{j \log_{b_2}(b_1)\}}, \cdots, b_k^{-\{j \log_{b_k}(b_1)\}}),
\end{align*}
where the last equal holds because $$b_i^{[j \log_{b_i}(b_1)]} b_i^{\{j \log_{b_i}(b_1)\}}=b_i^{j \log_{b_i}(b_1)}=b_i^{\log_{b_i}(b_1^j)}=b_1^j.$$

We define the rotation of angle $w:=(\log_{b_2}(b_1), \cdots, \log_{b_k}(b_1))$ in the $k-1$-dimensional torus, starting at $\vv{0}$.
Since the orbit of any rotation starting at $\vv{0}$ in any torus passes as close to $\vv{0}$ as we want infinitely many times, we find that
\[\{n\vv{w} \ \text{mod}(\Z^{k-1})\}_{n\geq 0}=(\{n\log_{b_2}(b_1)\}, \cdots, \{n\log_{b_k}(b_1)\})_{n\geq 0}\] passes as close as $\vv{0}$ as we want infinitely many times.

Let $\delta \in (0,\frac{1}{4})$ small enough. Let $\varepsilon>0$ be a small number so that $b_i^{-\varepsilon}>1-\delta$ for all $1\leq i \leq k$ (i.e.: $\varepsilon< \frac{|\log(1-\delta)|}{\log(\max_{1\leq i \leq k} b_i)}$) ($\varepsilon$ depends on $\max_{1\leq i \leq k} b_i$ and $\delta$). We will work with a subsequence of $n \in \N$ so that $\dist (n\vv{w}, \Z^{k-1})<\varepsilon$.
Then, writing $v_i:=b_i^{-\{n \log_{b_i}(b_1)\}}$ for $2\leq i \leq k$ we have \[a_n=b_1^n(1, b_2^{-\{n \log_{b_2}(b_1)\}}, \cdots, b_k^{-\{n \log_{b_k}(b_1)\}})=b_1^n(1, v_2, \cdots, v_k),\] where $1\geq v_i:=b_i^{-\{n \log_{b_i}(b_1)\}}>b_i^{-\varepsilon}>1-\delta$ (so multiplying by $v_i$ represents a small perturbation). Since $b_i^{[n \log_{b_i}(b_1)]} b_i^{\{n \log_{b_i}(b_1)\}}=b_1^n$, by definition of $v_i$ we get \[b_i^{[n \log_{b_i}(b_1)]}=b_1^n v_i \text{ with } v_i \in (1-\delta, 1].\] Because of the observation before this theorem: For each $A_{b_i}$, $A_{b_i}\cap (a_{i,n}, b_i a_{i,n}]$ contains the set \[b_i^{[n\log_{b_i}(b_1)]+1} \mathcal{C}_{b_i,0}=:b_i \mathcal{C}_{b_i,[n\log_{b_i}(b_1)]}.\] This set has thickness $b_i-2$, convex hull $b_i^{[n\log_{b_i}(b_1)]}[\frac{b_i}{b_i-1},b_i]=b_1^n v_i [\frac{b_i}{b_i-1},b_i]$, length of largest gaps equal to $b_i b_1^n v_i \frac{1}{b_i-1}$ and length of remaining intervals while removing those gaps equal to $b_i b_1^n v_i(1-\frac{1}{b_i-1})$.

When $M$ is large, $\frac{b_i}{b_i-1}$ is very close to $1$ ($2\leq i \leq k$), and we know that $v_i$ is a small perturbation, so intuitively all the left endpoints of these sets are very close to $b_1^n$. This is not true for the right endpoints, so we need to erase some part at the right of these sets but without a drop in their thickness. An easy way to erase the right part of a set without a drop of the thickness, is erasing everything at the right of a point in a gap of first level, this means the largest gaps (in this way we keep the relationship between bridges and gaps appearing in the definition of thickness, but we have fewer quotients).

The right endpoints are $b_1^n v_i b_i$ where $v_i$ is close to $1$ and $b_1< \cdots < b_k$. We will erase a part of the set at the right, close to the endpoints $b_1^n v_i b_i$ for $2\leq i \leq k$ so that the new right endpoints are close to $b_1^{n+1}$ (which is the right endpoint of $b_1 C_{b_1,n}$).

Notice that $A_{b_i}\cap b_1^n v_i [\frac{b_i}{b_i-1}, b_i]\supset b_1^n v_i b_i C_{b_i,0}$ (which is close to $b_1^n [1,b_i]\supseteq b_1^n [1, b_1+1]$). The set $b_1^n v_i b_i C_{b_i,0}$ has convex hull $b_1^n v_ib_i[\frac{1}{b_i-1},1]$. Since the largest gaps in $\mathcal{C}_{b_i,0}$ have length $\frac{1}{b_i(b_i-1)}$, the largest gaps in $b_1^n v_i b_i C_{b_i,0}$ have length $|G|=b_1^n \frac{v_i}{b_i-1}$ and the intervals of the first level of construction (after removing the largest gaps) have length $|I|\leq b_1^n (\frac{1}{b_i}-\frac{1}{b_i(b_i-1)})b_i v_i= b_1^n v_i \frac{b_i-2}{b_i-1}$, we see that \[|G|+|I|\leq b_1^n v_i\leq b_1^n.\]

Given an interval $G$ we define $\ell (G)$ to be the left endpoint of $G$.

Let's consider $G_i$ the gap of first level of $b_1^n v_i b_i C_{b_i,0}$ so that the distance from $\ell (G_i)$ to $b_1^{n+1}$ is $\leq \frac{|G_i|+|I_i|}{2}\leq \frac{b_1^n}{2}$. We define $R_{i,n}:=\ell (G_i)$, and restrict $b_1^n v_i b_i C_{b_i,0}$ to the left of $R_{i,n}$.

We define \[B_{i,n}:= A_{b_i}\cap [b_1^n v_i \frac{b_i}{b_i-1}, R_{i,n}].\]
Let's denote $L_{i,n}:=b_1^n v_i \frac{b_i}{b_i-1}$.

We know that $R_{i,n}\in [b_i^{n+1} -\frac{b_i^{n}}{2}, b_i^{n+1} +\frac{b_i^{n}}{2}]$, and $L_{i,n}\in [b_i^{n} v_i, b_i^{n} \frac{4}{3}]$.

We have:
\begin{enumerate}
\item $\tau (B_{i,n})\geq b_i -2$ for every $1\leq i \leq k$,
\item We define $B:=\bigcap_{1\leq i \leq k} \text{conv} (B_{i,n})$. Since $R_{i,n}\geq b_1^{n+1}-\frac{b_1^n}{2}$, $L_{i,n}\leq b_1^n \frac{4}{3}$ for every $2\leq i \leq k$ and $b_1\geq 4$, we have \[|B|\geq b_1^{n+1}-b_1^n \frac{1}{2}-b_1^n \frac{4}{3}\geq b_1^{n+1} \frac{13}{24}.\]
\item Using that $b_1\geq 4$, we get
\begin{align*}
\max_{1\leq i \leq k} \diam (\text{conv} (B_{i,n}))&\leq \max_{1\leq i \leq k} b_1^{n+1}+\frac{b_1^n}{2}-b_1^n v_i\\
&\leq b_1^{n+1}+\frac{b_1^n}{2}\leq b_1^{n+1} (1+\frac{1}{2b_1})\leq b_1^{n+1}\frac{9}{8}
\end{align*}
\end{enumerate}

Let's see that if $b_1, \cdots, b_k \geq M$ then 
$\bigcap_i B_{i,n}\neq \emptyset$.

Let's apply the Theorem \ref{thm: FY} with $d=1$ to $C_i:=B_{i,n}$.
We already know that the two first hypotheses hold. Let's see the third one.

Since
$\frac{\diam (B)}{\sup_i \diam (C_i)}\geq \frac{b_1^{n+1} \frac{13}{24}}{b_1^{n+1} \frac{9}{8}}=\frac{13}{27}\geq \frac{1}{4}$, we see that $\beta=\frac{1}{4}$.

We need to show that there exists $c\in (0,1)$ so that \[\sum_{1\leq i\leq k} (b_i-2)^{-c}\leq \frac{1}{432^2}(\frac{1}{4})^c (1-(\frac{1}{4})^{1-c}).\]

Taking $c:=1-\frac{1}{\log(\frac{b_1-2}{4})}$,  since $b_1\leq b_i$ for every $1\leq i \leq k$, it is enough to show that $k (432)^2 \leq \frac{b_1-2}{4} \frac{1}{e}(1-(1/4)^{1/\log(\frac{b_1-2}{4})})$.

Since $f:[3,\infty)\to \R$ defined by $f(x):=\log(x)(1-(1/4)^{1/\log(\frac{x}{4})})$ is a decreasing function converging to $\log(4)$ when $x\to\infty$, considering $x=b_1-2$, one can see that it is enough to show the following:
\[k \frac{e 4 (432)^2}{\log (4)} \leq \frac{b_1-2}{\log(b_1-2)}.\]

Since $g_k:[3,\infty)\to \R$ defined by $g_k(x):=\frac{x}{\log(x)}-k \frac{e 4 (432)^2}{\log (4)}$ is an increasing function with $g_k(3)<0$ and $\lim_{x\to \infty} g_k(x)=\infty$, there is $M-2 \geq 3$ so that $g(M-2)>0$ (and we can take the smallest $M \in \N_{\geq 5}$ satisfying that). Then, for every $b_1\geq M=M(k)$, we have $k (432)^2 \leq \frac{b_1-2}{4} \frac{1}{e}(1-(1/4)^{1/\log(\frac{b_1-2}{4})})$.

This finishes the proof. If we want to get a more explicit $M$ which would work, we can consider, for example, $b_1-2=x=k \log^{1+\varepsilon}(k)$ (i.e: $M=k \log^{1+\varepsilon}(k)+2$), because $g_k(x)=\frac{k \log^{1+\varepsilon}(k)}{\log(k \log^{1+\varepsilon}(k))}-k \frac{e 4 (432)^2}{\log (4)}$ is positive considering $k$ large enough.
\end{proof}

\section{Omitting multiple separated digits}\label{sec: omit nonzero}

We have considered sets of integers omitting the digit zero (in various bases simultaneously). Now we want to consider the case when the missing digit(s) may not be zero. Our topological argument is sensitive to the choice of digits. Other than that, the argument is very similar to the proof of Theorem \ref{thm: main1}.

\begin{proof}[Proof of Theorem \ref{thm: main2}]
We can assume that $b_1< \cdots < b_k$. Let $b \in \N_{\geq 3}$, we define $D:=\{0, \cdots, b-1\}\setminus D^C$ (the set of digits to use) and the compact set
\[\mathcal{C}_{b,D}:=\{x \in [0,1]: \ x=\sum_{i<0}c_i b^i \text{ with } c_i \in D\}.\]

One can see that the convex hull is $[\frac{1}{b-1},1]$ (by using the geometric series, and the fact that $0 \notin D$ and $b-1 \in D$). Gaps of level $k$ (Newhouse's structure) have length $|G_k|=\frac{1}{b^k (b-1)}$ with remaining intervals of length $|I_k|=\frac{r-\frac{1}{b(b-1)}}{b^{k-1}}$. So, $\mathcal{C}_{b,D}$ has thickness is $=(b-1)r-\frac{1}{b}\geq r(b-2)$.

We also define for every $j\in \N$ its expanded version by $b^j$ as \[\mathcal{C}_{b, D, j}:=b^{j+1} \mathcal{C}_{b,D}.\]
Its convex hull is $b^{j+1} [\frac{1}{b-1},1]$, its thickness is $\geq r(b-2)$.

Note that for a fixed $b$, the sets $\mathcal{C}_{b,D,j}$ are disjoint and moving to the right (when $j$ moves to infinity). Then, the set of positive numbers that can be written in base $b$ with digits in $D$ is
\begin{align*}A_{b, D}&:=\{x \in (0, \infty): \ x=\sum_{i<j}c_i b^i \text{ with } c_i \in D \text{ for some } j \in \N\}\\ &= \bigcupdot_{j \in \N} \mathcal{C}_{b, D, j}.\end{align*}
For each $i \in \{1, \cdots, k\}$ we define $A_{b_i, D_i}$ as before, and
\[A:=\bigcap_{1\leq i \leq k}A_{b_i, D_i},\] the set of positive numbers that can be written in bases $b_1, \cdots, b_k$ with digits in $D_i$ respectively.

As before, it is enough to prove that $A$ is unbounded. For this, we will get a sequence of  $k$-tuples of compact subsets $(B_{1,n}, \cdots, B_{k,n})_n$, where $B_{i,n}\subseteq A_{b_i}$, the sets $B_{i,n}$ are thick and the convex hulls of sets in a tuple are not so different (in this way we will ensure a non-empty intersection, i.e: $\exists \, x_n \in \bigcap_{1\leq i\leq k} B_{i,n} \subseteq A$), but the sequence of tuples of sets that are alike will be moving to infinity.

We define a sequence of $k$-tuples $(a_j)_{j \in \N}$ where $a_j:=(a_{1,j}, \cdots, a_{k,j})$ as follows:
\begin{align*}
a_j&:=(b_1^j, b_2^{[j \log_{b_2}(b_1)]}, \cdots, b_k^{[j \log_{b_k}(b_1)]})\\
&=b_1^j(1, b_2^{-\{j \log_{b_2}(b_1)\}}, \cdots, b_k^{-\{j \log_{b_k}(b_1)\}}),
\end{align*}
where the last equal holds because $$b_i^{[j \log_{b_i}(b_1)]} b_i^{\{j \log_{b_i}(b_1)\}}=b_i^{j \log_{b_i}(b_1)}=b_i^{\log_{b_i}(b_1^j)}=b_1^j.$$

We define the rotation of angle $w:=(\log_{b_2}(b_1), \cdots, \log_{b_k}(b_1))$ in the $k-1$-dimensional torus, starting at $\vv{0}$.
Since the orbit of any rotation starting at $\vv{0}$ in any torus passes as close as $\vv{0}$ as we want infinitely many times, we have that
\[\{n\vv{w} \ \text{mod}(\Z^{k-1})\}_{n\geq 0}=(\{n\log_{b_2}(b_1)\}, \cdots, \{n\log_{b_k}(b_1)\})_{n\geq 0}\] passes as close as $\vv{0}$ as we want infinitely many times.

Let $\delta \in (0,\frac{1}{4})$ small enough. Let $\varepsilon>0$ be a small number so that $b_i^{-\varepsilon}>1-\delta$ for all $1\leq i \leq k$ (i.e.: $\varepsilon< \frac{|\log(1-\delta)|}{\log(\max_{1\leq i \leq k} b_i)}$) ($\varepsilon$ depends on $\max_{1\leq i \leq k} b_i$ and $\delta$).

We will work with a subsequence of $n \in \N$ so that $\dist (n\vv{w}, \Z^{k-1})<\varepsilon$.
Then, writing $v_i:=b_i^{-\{n \log_{b_i}(b_1)\}}$ for $2\leq i \leq k$ we have that \[a_n=b_1^n(1, b_2^{-\{n \log_{b_2}(b_1)\}}, \cdots, b_k^{-\{n \log_{b_k}(b_1)\}})=b_1^n(1, v_2, \cdots, v_k),\] where $1\geq v_i:=b_i^{-\{n \log_{b_i}(b_1)\}}>b_i^{-\varepsilon}>1-\delta$ (so multiplying by $v_i$ represents a small perturbation).

Since $b_i^{[n \log_{b_i}(b_1)]} b_i^{\{n \log_{b_i}(b_1)\}}=b_1^n$, by definition of $v_i$ we get \[b_i^{[n \log_{b_i}(b_1)]}=b_1^n v_i \text{ with } v_i \in (1-\delta, 1].\]

Each $A_{b_i, D_i}$ contains in $b_i^{[n\log_{b_i}(b_1)]}[\frac{b_i}{b_i-1},b_i]$ the set \[b_i^{[n\log_{b_i}(b_1)]+1} \mathcal{C}_{b_i,D_i}=:b_i b_1^n v_i \mathcal{C}_{b_i,D_i}\] which has: thickness $\geq r(b_i-2)$, convex hull $b_i^{[n\log_{b_i}(b_1)]}[\frac{b_i}{b_i-1},b_i]=b_1^n v_i [\frac{b_i}{b_i-1},b_i]$.

We cut each set at the right at $R_{i,n}$ as in the last section, in a gap of first level associated to the avoidance of the digit $0$ (which is the largest one). Here we used that $0 \notin D_i$, so holes of sets from the last section are contained in holes of sets from this section, and we can use them to cut the sets at the right without a drop in their thickness.
We define \[B_{i,n}:= A_{b_i}\cap [b_1^n v_i \frac{b_i}{b_i-1}, R_{i,n}].\]
Let's denote $L_{i,n}:=b_1^n v_i \frac{b_i}{b_i-1}$.

We know that $R_{i,n}\in [b_i^{n+1} -\frac{b_i^{n}}{2}, b_i^{n+1} +\frac{b_i^{n}}{2}]$, and $L_{i,n}\in [b_i^{n} v_i, b_i^{n} \frac{4}{3}]$.

We have:
\begin{enumerate}
\item $\tau (B_{i,n})\geq r(b_i -2)$ for every $1\leq i \leq k$,
\item We define $B:=\bigcap_{1\leq i \leq k} \text{conv} (B_{i,n})$. Since $R_{i,n}\geq b_1^{n+1}-\frac{b_1^n}{2}$, $L_{i,n}\leq b_1^n \frac{4}{3}$ for every $2\leq i \leq k$ and $b_1\geq 4$, we have \[|B|\geq b_1^{n+1}-b_1^n \frac{1}{2}-b_1^n \frac{4}{3}\geq b_1^{n+1} \frac{13}{24}.\]
\item Using that $b_1\geq 4$, we get
\begin{align*}
\max_{1\leq i \leq k} \diam (\text{conv} (B_{i,n}))&\leq \max_{1\leq i \leq k} b_1^{n+1}+\frac{b_1^n}{2}-b_1^n v_i\\
&\leq b_1^{n+1}+\frac{b_1^n}{2}\leq b_1^{n+1} (1+\frac{1}{2b_1})\leq b_1^{n+1}\frac{9}{8}
\end{align*}
\end{enumerate}

Let's see that if $b_1, \cdots, b_k \geq M$ then
$\bigcap_i B_{i,n}\neq \emptyset$.

Let's apply the Theorem \ref{thm: FY} to $C_i:=B_{i,n}$.
We already know that the two first hypotheses hold. Let's see the third one.

Since
$\frac{\diam (B)}{\sup_i \diam (C_i)}\geq \frac{b_1^{n+1} \frac{13}{24}}{b_1^{n+1} \frac{9}{8}}=\frac{13}{27}\geq \frac{1}{4}$. Then $\beta=\frac{1}{4}$.

We need to show that there exists $c\in (0,1)$ so that \[\sum_{1\leq i\leq k} (r(b_i-2))^{-c}\leq \frac{1}{432^2}(\frac{1}{4})^c (1-(\frac{1}{4})^{1-c}).\]

Taking $c:=1-\frac{1}{\log(\frac{r(b_i-2)}{4})}$,  since $b_1\leq b_i$ for every $1\leq i \leq k$, it is enough to show that $k (432)^2 \leq \frac{r(b_i-2)}{4} \frac{1}{e}(1-(1/4)^{1/\log(\frac{r(b_i-2)}{4})})$.

Since $f:[3,\infty)\to \R$ defined by $f(x):=\log(x)(1-(1/4)^{1/\log(\frac{x}{4})})$ is a decreasing function converging to $\log(4)$ when $x\to\infty$, considering $x=r(b_i-2)$, one can see that it is enough to show the following:
\[k \frac{e 4 (432)^2}{\log (4)} \leq \frac{r(b_i-2)}{\log(r(b_i-2))}.\]

Since $g_k:[3,\infty)\to \R$ defined by $g_k(x):=\frac{x}{\log(x)}-k \frac{e 4 (432)^2}{\log (4)}$ is an increasing function with $g_k(3)<0$ and $\lim_{x\to \infty} g_k(x)=\infty$, there is $r(M-2) \geq 3$ so that $g(r(M-2))>0$ (and we can take the smallest $M \in \N_{\geq 3}$ satisfying that).

Then, for every $b_1\geq M=M(k)$, we have $k (432)^2 \leq \frac{r(b_1-2)}{4} \frac{1}{e}(1-(1/4)^{1/\log(\frac{r(b_1-2)}{4})})$.

If we want to get a more explicit $M$ which would work, we can consider, for example, $r(b_1-2)=x=k \log^{1+\varepsilon}(k)$ (i.e: $M=k \log^{1+\varepsilon}(k)\frac{1}{r}+1$), because $g_k(x)=\frac{k \log^{1+\varepsilon}(k)}{\log(k \log^{1+\varepsilon}(k))}-k \frac{e 4 (432)^2}{\log (4)}$ is positive considering $k$ large enough.
\end{proof}

As a direct Corollary, we have a general result to omit thick  digits sets that may or may not contain $0$.

\begin{corollary}
Let $k\geq 2$ be an integer, $(b_1, \cdots, b_k)$ be a $k$-tuple of natural numbers and $r\in (0,1)$. Let $D^C_1, \cdots , D^C_k$ be sets satisfying:
\begin{itemize}
\item $D^C_i$ does not have consecutive numbers,
\item $\{0, \cdots , b_i-2\}\setminus D^C_i$ is formed by rows of at least $b_i r$ consecutive numbers, and in case $0 \notin D^C_i$ the row of digits containing $0$ is of length at least $b_i r +1$.
\end{itemize}
 Let $M=M(k)\in \N_{\geq 4}$ be the smallest number satisfying that \[g_k(r(M-2)):=\frac{r(M-2)}{\log(r(M-2))}-k \frac{e 4 (432)^2}{\log (4)}>0.\] If $b_i \geq M=M(k)$ for every $1\leq i \leq k$, then there are infinitely many integers whose base $b_1, \cdots, b_k$ expansions omit the digits in the sets $D^C_1, \cdots, D^C_k$ respectively.
\end{corollary}

One can see this by considering $\tilde{D^C_i}:=D^C_i \cup \{0\}$ and applying the previous Theorem.
Since $A_{b_i,\tilde{D_i}} \subseteq A_{b_i,D_i}$, then $\bigcap_{1\leq i \leq k} A_{b_i,\tilde{D_i}} \subseteq \bigcap_{1\leq i \leq k} A_{b_i,D_i}$ and the Theorem ensures that $\bigcap_{1\leq i \leq k} A_{b_i,\tilde{D_i}}$ is unbounded, we have that $\bigcap_{1\leq i \leq k} A_{b_i,D_i}$ is unbounded as well, so there are infinitely many integers whose base $b_1, \cdots, b_k$ expansions omit the digits in the sets $D^C_1, \cdots, D^C_k$ respectively.

\section{Missing digits sets with respect to algebraic bases}\label{sec: algebraic}
Our method can be generalised to consider various missing-digit sets in an algebraic setting. We will use Theorem \ref{thm: FY}. The method will be roughly the same as in the proofs of our previous theorems. The most difficult part is perhaps the introduction of the context. Given this fact, the actual proofs in this section are first given for specific examples. After that, it is more or less straightforward to obtain the general proof. We will indicate only major modifications.

We discuss the algebraic setting. Let $\B{K}$ be a number field. Let $O_{\B{K}}$ be its ring of integers. Assume that $O_{\B{K}}$ is an Euclidean domain. For each non-unit $b\in O_{\B{K}}$, we can define the base expansion w.r.t to $b$. Consider the finite group $O_{\B{K}}/bO_{\B{K}}.$ We can choose any representation $D_b$ of this finite group in $O_{\B{K}}.$ After that, for each integer $z\in O_{\B{K}},$ we can perform the Euclidean algorithm and obtain integers $z'\in O_{\B{K}},r\in D_b,$
\[
z=z'b+r.
\]
This step does not need the fact that $O_{\B{K}}$ is Euclidean. What we need is that, for some choice of $D_b$, we can continue to perform the above arithmetic and after at most finitely many steps, we will arrive at a unit $u$. Any choice of such $D_b$ can be regarded as digits in base $b$ expansion. The above iterated arithmetic gives us a representation
\[
z=ub^k+\sum_{j=0}^{k-1}r_j b^j
\]
for a unit $u,$ a rational integer $k$ and digits $r_0,\dots,r_{k-1}.$ Here, $u$ may not be in $D_b$. We require that $k$ is the first time this unit $u$ appears. We can assume that $1\in D_b$ and consider the set of integers $Z_b$ with $u=1$ in the above expansion. Notice that $Z_b$ may not be the whole integer set. This is already the case when $\B{K}=\B{Q}$, $b\geq 2,$ $D_b=\{0,1,\dots,b-1\}$ and then $Z_b$ are non-negative integers. A slightly more non-trivial example is when $\B{K}=\B{Q}[i]$, $b=1+i$ and $D_b=\{0,1\}.$ In this case $Z_b$ is a subset of $O_{\B{K}}$. The integers with at most $15$ digits are plotted in Figure \ref{fig: 15digits}.
\begin{figure}[h]
\includegraphics[width=5cm]{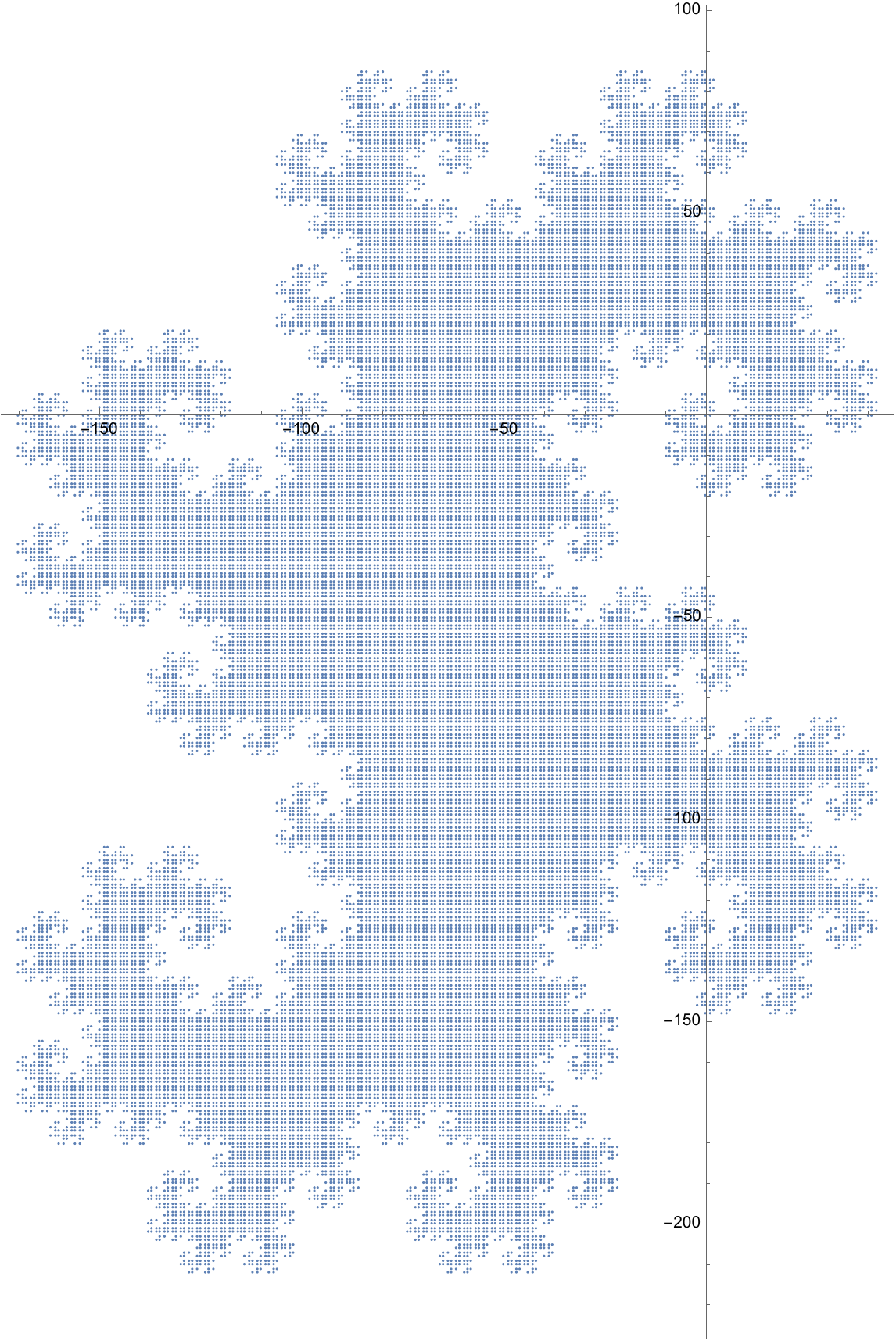}
\caption{Integers in $\mathbb{Z}[i]$ in base $b=1+i$, $D_b=\{0,1\}$ with at most 15 digits.}
\label{fig: 15digits}
\end{figure}
Next, instead of integers, we can also consider numbers of the form
\[
z=\sum_{j< 0}r_j b^j.
\]
They may not be in $\B{K}$ unless the expansion is eventually periodic. However, we can interpret them as elements in the complete field $\B{K}_\infty.$ We can now choose a subset $D\subset D_b$ and consider
\[
\sum_{j=0}^{k-1}r_j b^j, \sum_{j< 0}r_j b^j
\]
for $r_j\in D$. The latter set is naturally denoted as $\mathcal{C}_{b,D}$ which is consistent with our notation for the case $\B{K}=\B{Q}.$ Similarly, as in the Proof of Theorem \ref{thm: main2}, we can also construct missing-digit fractals $\mathcal{C}_{b,D,j}$ and  $A_{b,D}.$  For example, if $b=1+2i$, we have $D_b=\{0,1,2,1+i,2+i\}.$ In this case, the geometry of $C_{b,D_b}$ is as in Figure \ref{fig: level6}. The set $C_{b,D_b}$ looks quite irregular. However, we claim that under the canonical quotient $\mathbb{C}/\mathbb{Z}[i]$(or $\mathbb{R}^2/\mathbb{Z}^2$), $C_{b,D_b}$ is a fundamental domain. This reflects the fact that $D_b$ is a full representation of $O_\B{K}/bO_\B{K}$. See Firgure \ref{fig: level6}.
\begin{figure}
\begin{subfigure}{0.49\linewidth}
\includegraphics[width=\linewidth]{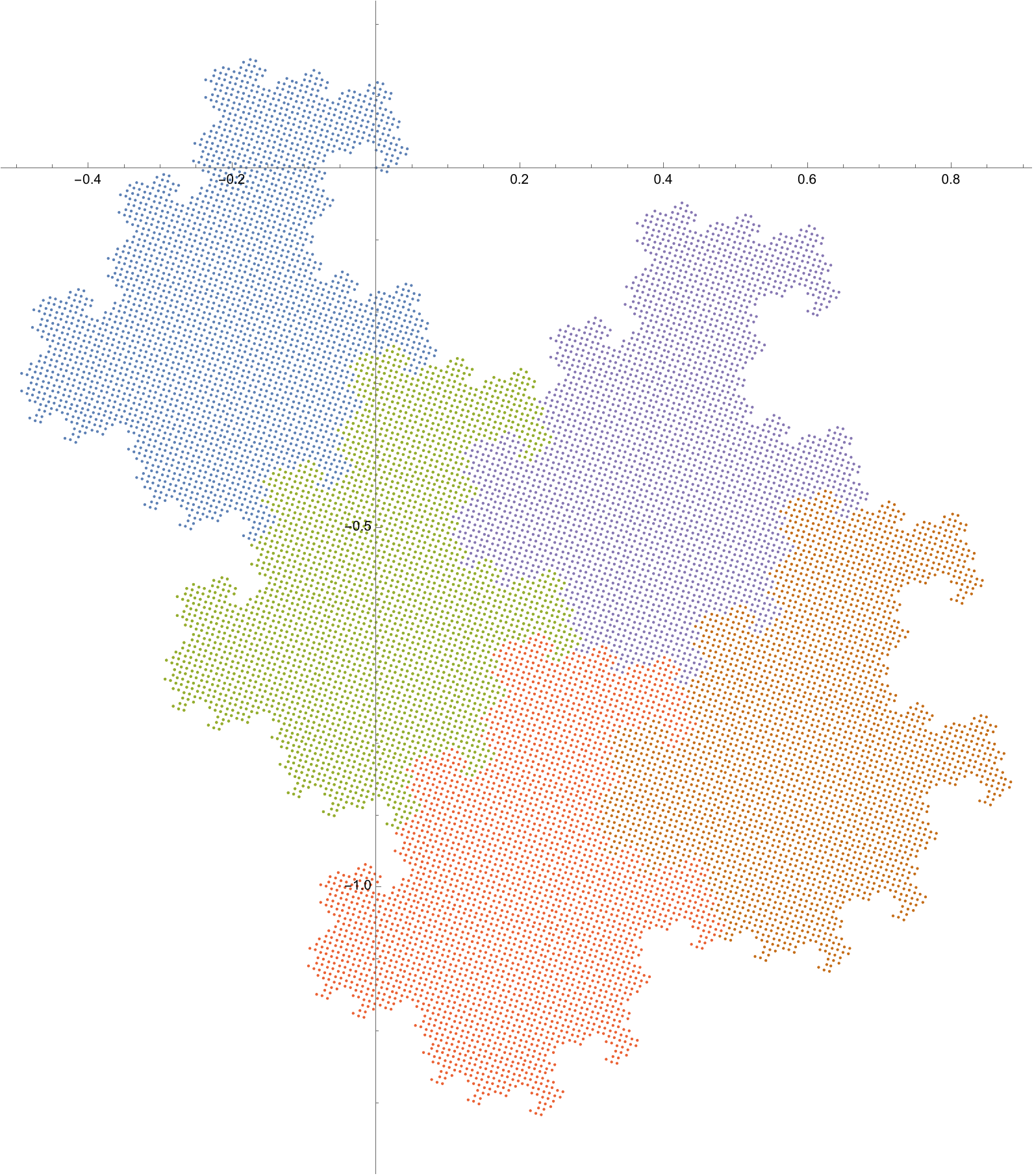}
\caption{$C_{b,D_b}$ for $b=1+2i$. Different digits are colored differently.}
\label{fig: f1}
\end{subfigure}
\begin{subfigure}{0.49\linewidth}
\includegraphics[width=\linewidth]{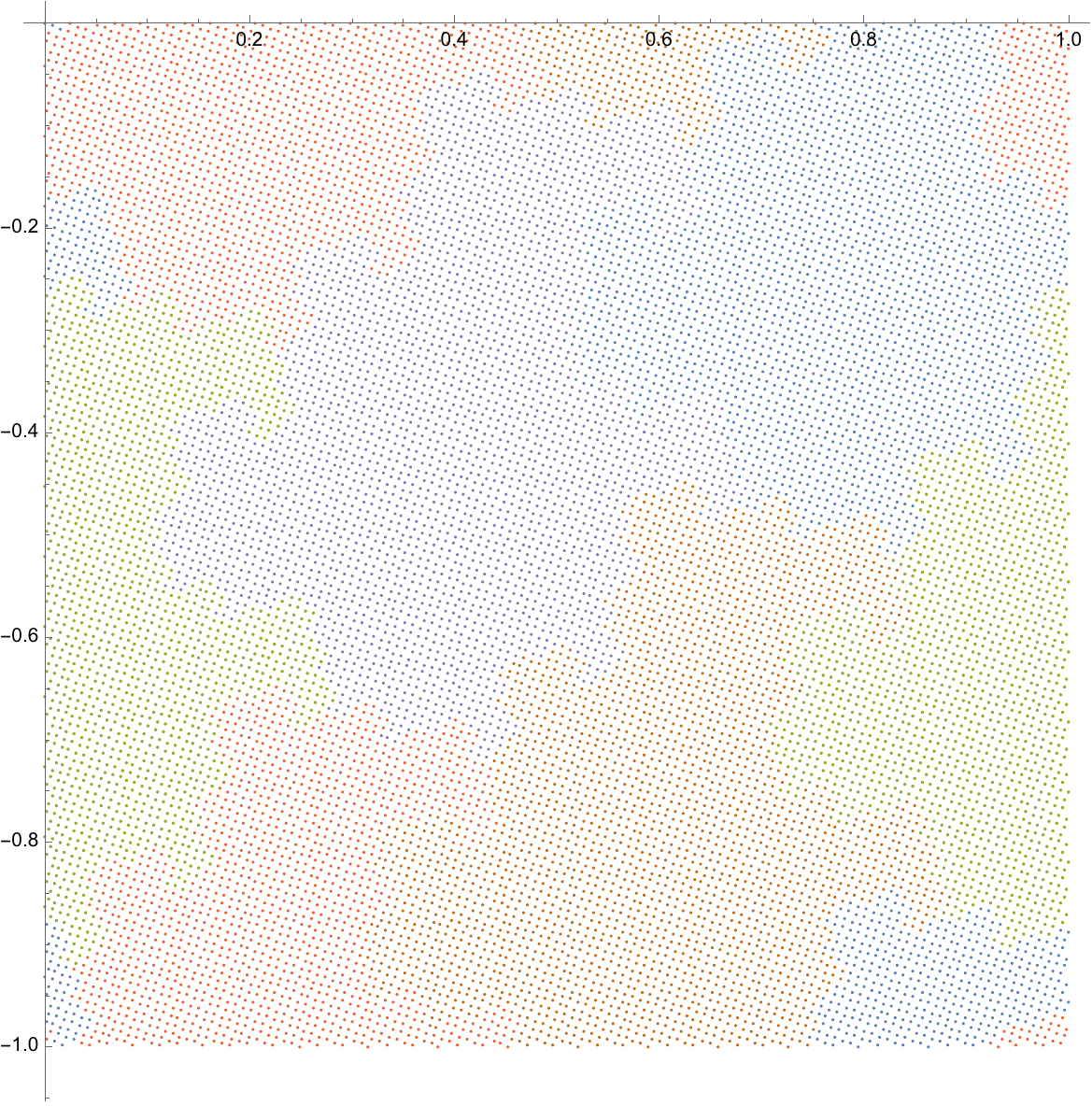}
\caption{$C_{b,D_b}$ is a fundamental domain of $\mathbb{C}/\mathbb{Z}[i]$.}
\label{fig: f2}
\end{subfigure}
\caption{}
\label{fig: level6}
\end{figure}

We now focus on the special case when $\B{K}=\B{Q}[i].$ We can naturally regard $\B{Q}[i]$ as a subset of $\mathbb{R}^2.$\footnote{More generally, we can use the Galois embeddings to regard $\B{K}\subset\mathbb{R}^s\times\mathbb{C}^t$ for suitable integers $s,t.$} Let $b\subset\mathbb{Z}[i]=O_{\mathbb{Q}[i]}$ be a non-unit integer. Observe that $b\mathbb{Z}[i]\subset \mathbb{Z}[i]$ is a sub-lattice of co-volume $N(b)$ where $N(.)$ is the norm of the number field $\B{Q}[i].$ In our case, it coincides with $|b|^2$ where $|.|$ is the natural Euclidean norm in $\mathbb{C}$ or $\mathbb{R}^2.$ Observe that if $N(b)$ is large enough, the geometric shape of a fundamental domain $b\B{Z}[i]$ can be made to be a very large square.\footnote{There are different choices of fundamental domains conjugated with each other via $SL_2(\mathbb{Z})$. There is one which is a square.} In this fundamental domain, there are exactly $N(b)$ many elements in $\B{Z}[i]$. We can take those elements as $D_b.$ The following result captures the thickness of missing-digit fractals originating from base $b$ expansions with one missing digit. See Figure \ref{fig: missingone} for $b=1+2i$ where we plotted $bC_{b,D}$ for all possible $D$ with one missing digit.

\begin{figure}
	\centering
	\begin{subfigure}{0.32\linewidth}
		\includegraphics[width=\linewidth]{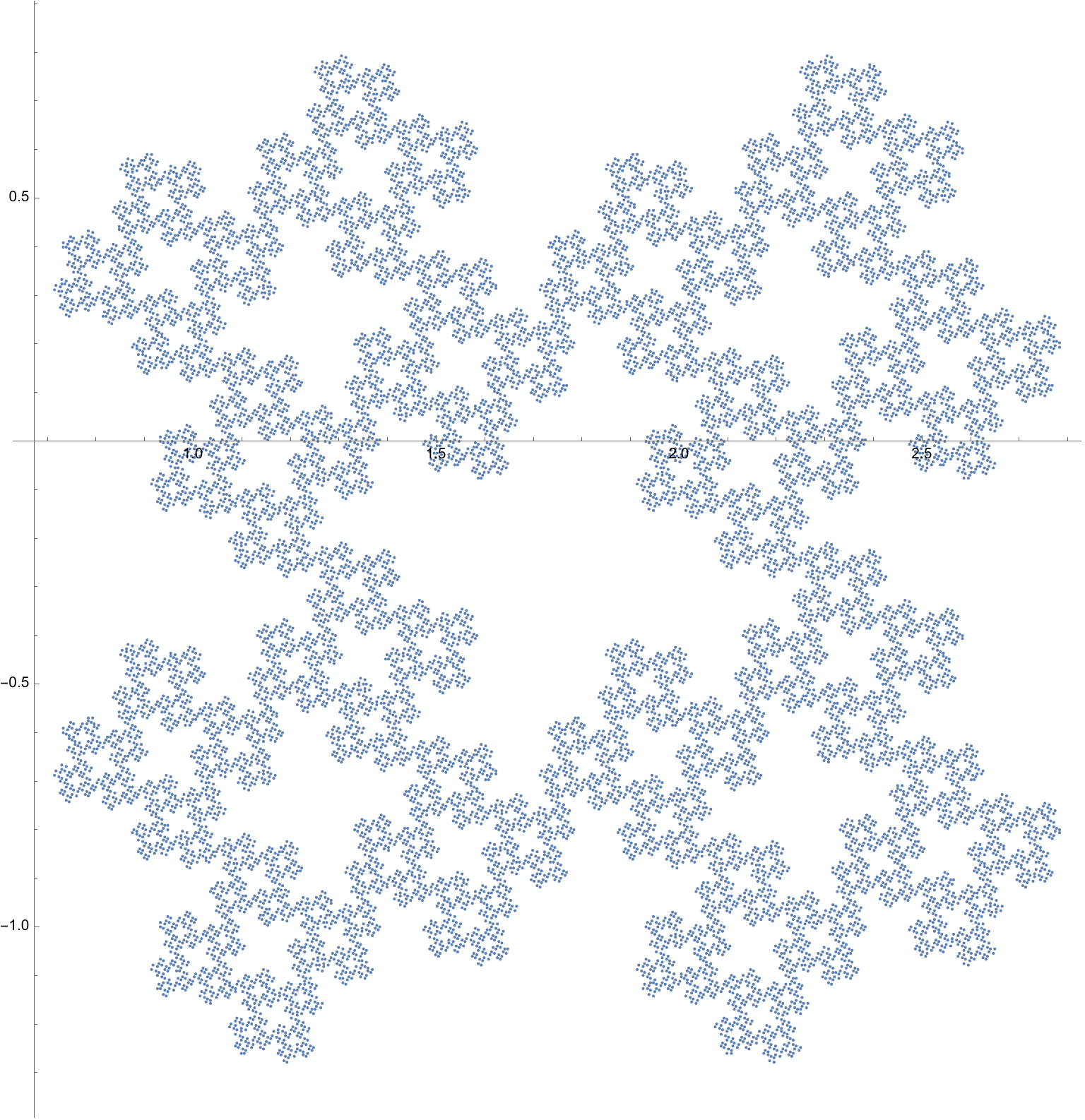}
		\caption{Missing $0$}
		\label{fig:subfigA}
	\end{subfigure}
	\begin{subfigure}{0.32\linewidth}
		\includegraphics[width=\linewidth]{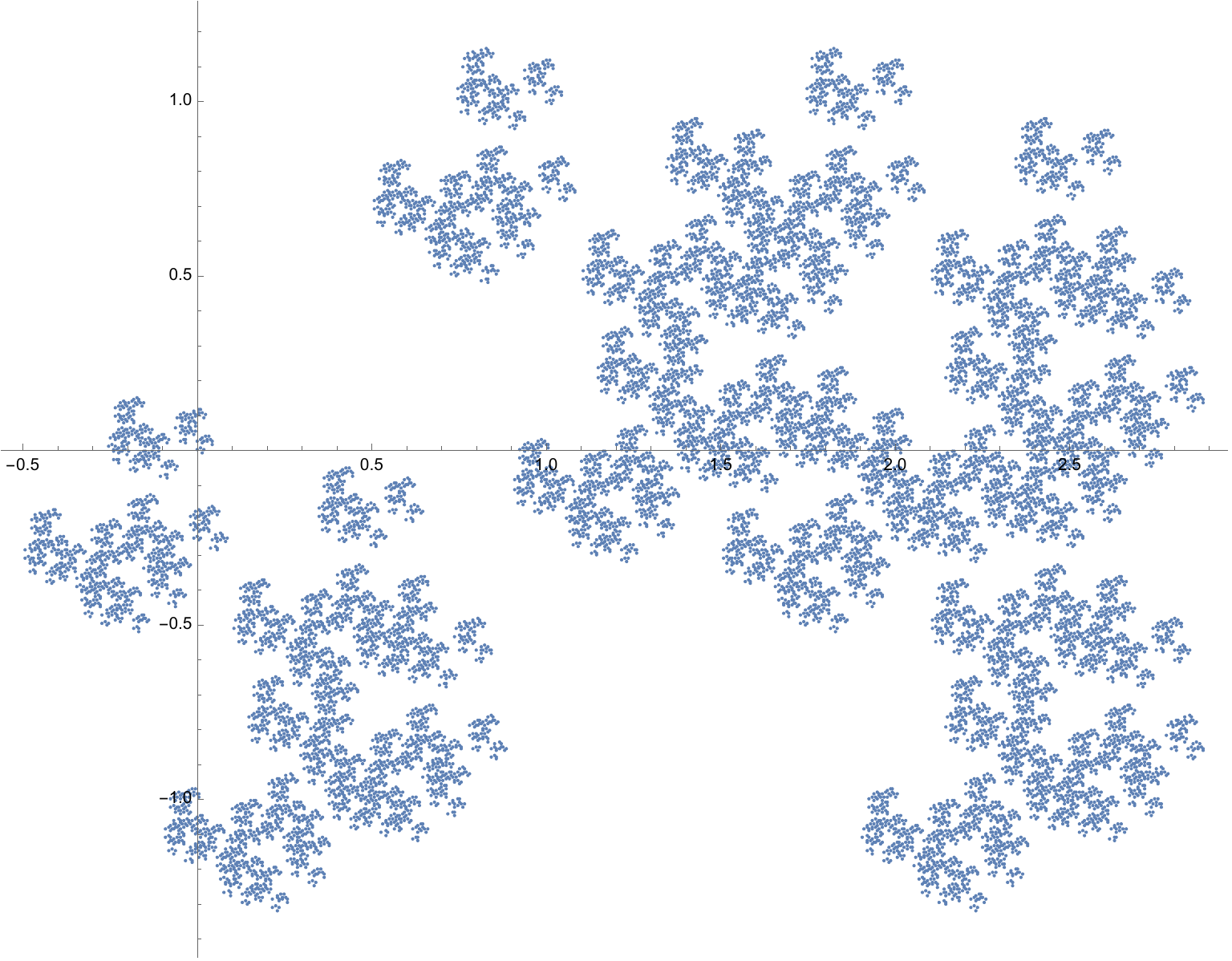}
		\caption{Missing $1$}
		\label{fig:subfigB}
	\end{subfigure}
	\begin{subfigure}{0.32\linewidth}
	        \includegraphics[width=\linewidth]{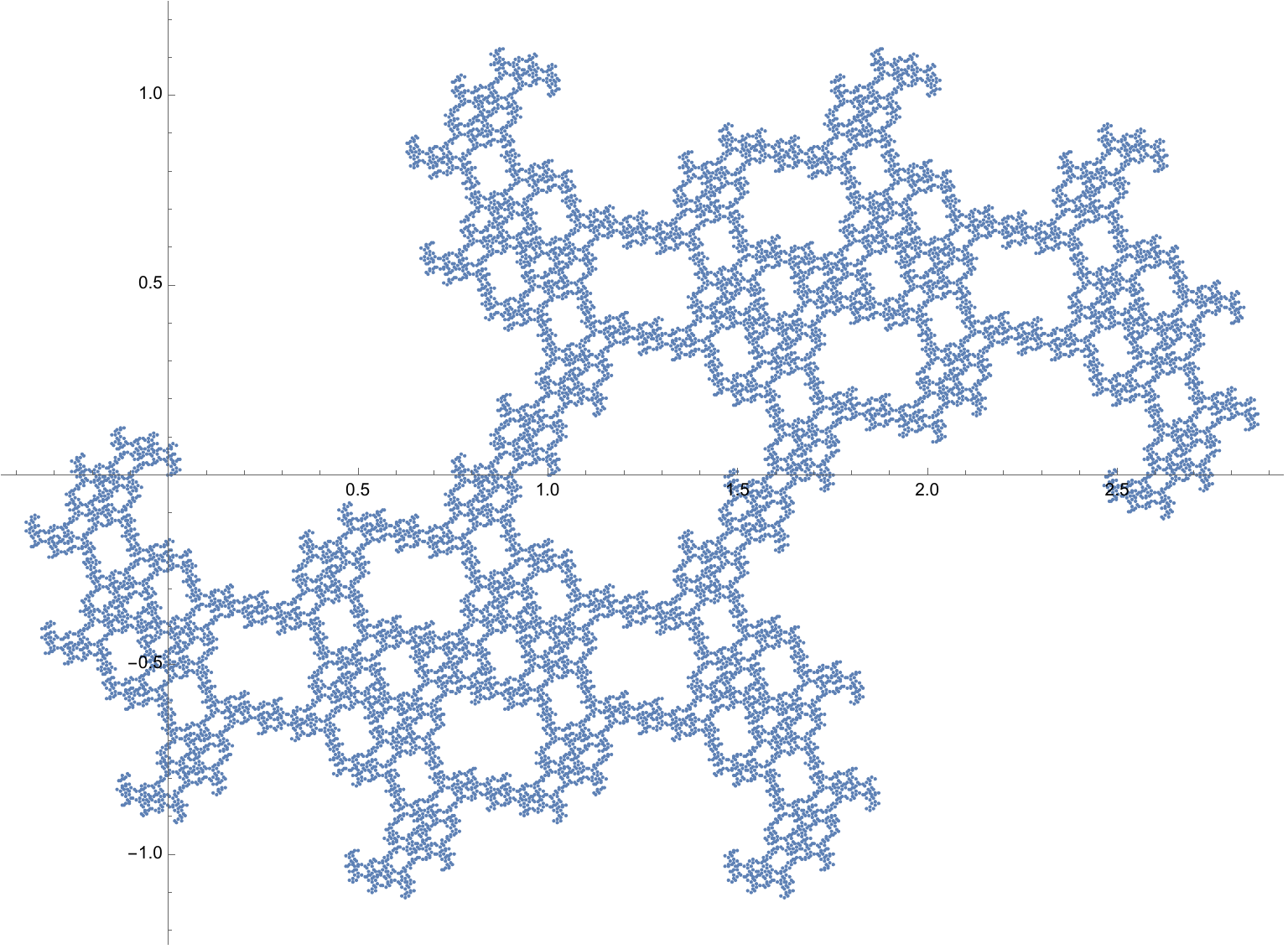}
	        \caption{Missing $2$}
	        \label{fig:subfigC}
         \end{subfigure}
         \begin{subfigure}{0.32\linewidth}
	        \includegraphics[width=\linewidth]{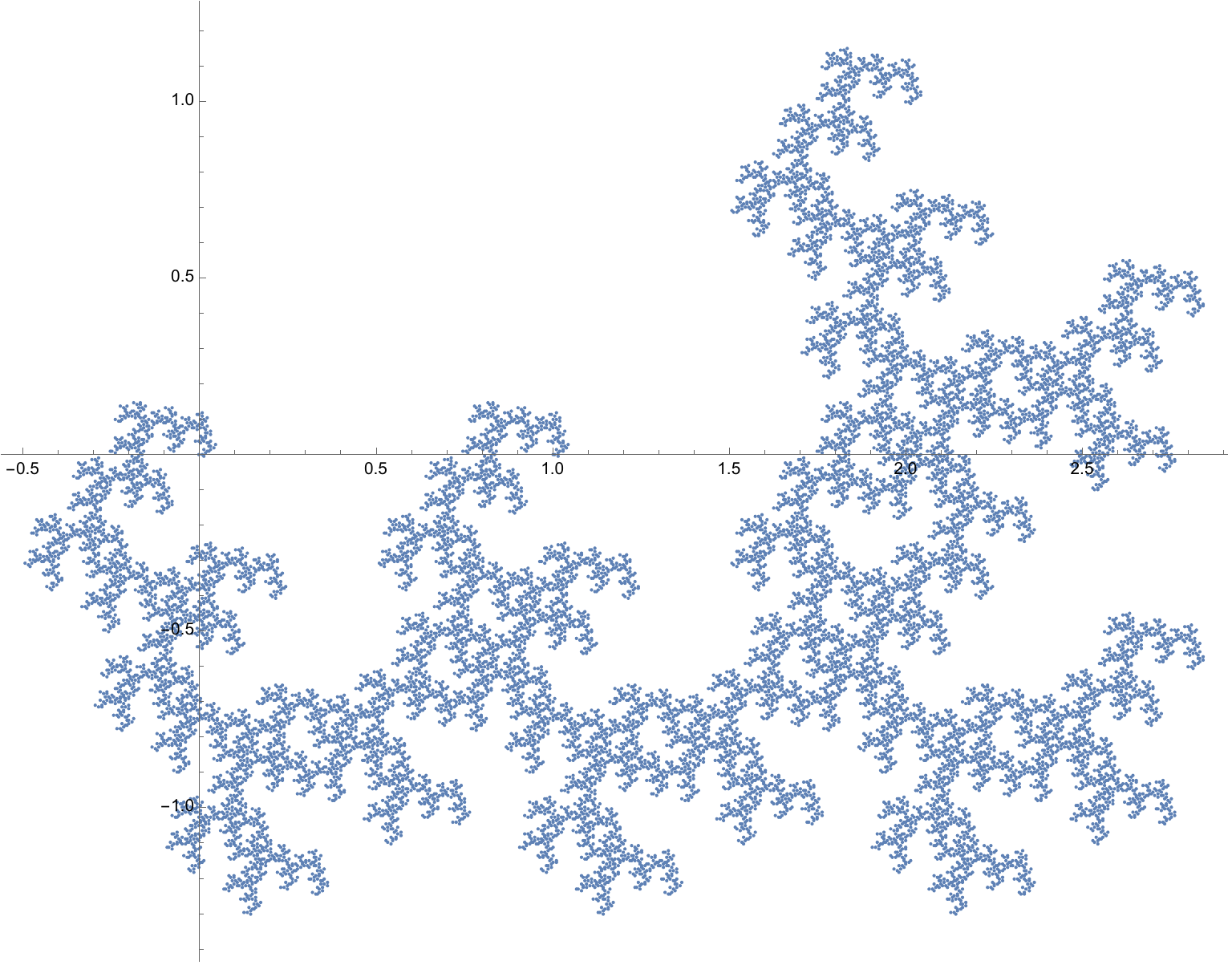}
	        \caption{Missing $1+i$}
	        \label{fig:subfigC}
         \end{subfigure}
         \begin{subfigure}{0.32\linewidth}
	        \includegraphics[width=\linewidth]{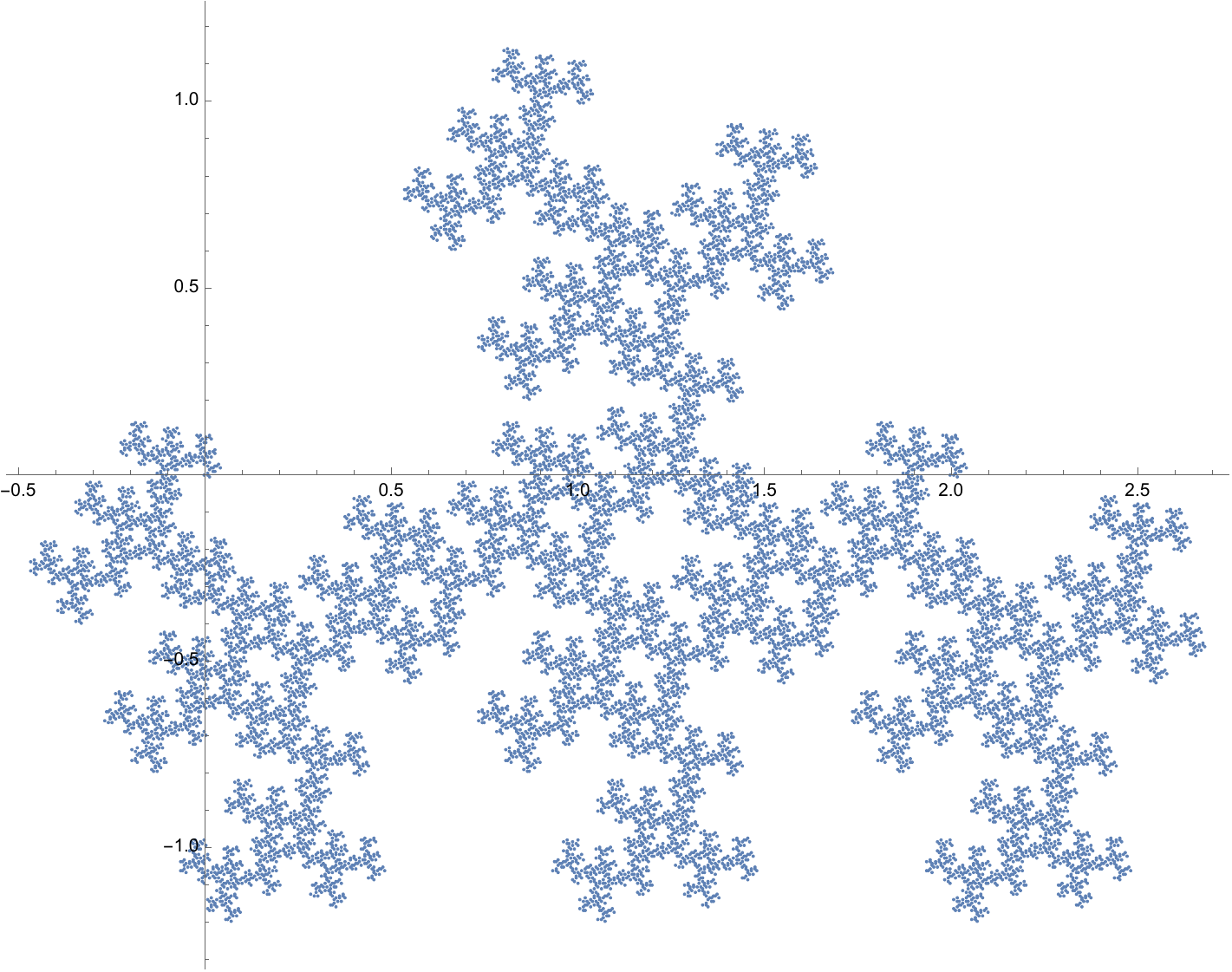}
	        \caption{Missing $2+i$}
	        \label{fig:subfigC}
         \end{subfigure}
	\caption{Base $b=1+2i$ with all possible one missing digit.}
	\label{fig: missingone}
\end{figure}

\begin{theorem}
    Fix $T\in (0,\infty).$ There is some $M>0$ such that all $C_{b,D}$ with $\#D=N(b)-1$ has thickness at least $T$ as long as $N(b)$ is large enough.
\end{theorem}
\begin{remark}
    An illustrative example is the Sierpinski Carpet based in a large integer $p.$ It is constructed as $b=p$ with one missing digit.
\end{remark}
\begin{proof}
    First, if $N(b)$ is large enough, the set $C_{b,D_b}$ is very close to being a full square. In fact, consider the square $S_b$ with side length one and rotated according to $b^{-1}$ (but not scaled). It can be checked that the convex hull of $C_{b,D_b}$ is of distance $O(1/N^{1/2}(b))$ to $S_b$ in terms of Hausdorff the metric. With one missing digit, $C_{b,D}$ is $1/N^{1/2}(b)$-close to $C_{b,D_b}.$ Notice that $C_{b,D}$ is self-similar whose contraction ratio is $1/N^{1/2}(b)$. From here, we see that the thickness of $C_{b,D}$ is at least
    $
    c_d N^{1/2}(b)
    $
    for some constant $c_d>0.$
\end{proof}

The thickness estimate holds for all $C_{b,D,j}$ because of the self-similarity. However, we cannot directly use Theorem \ref{thm: FY} to obtain intersections between different missing digit sets without checking the conditions (2,3). 

Consider different integers $b_1,\dots,b_k\in\mathbb{Z}[i]$ with large enough norms so that Theorem \ref{thm: FY} applies. Consider the logarithmic map
\[
\log: a\in \mathbb{Z}[i] \to (\log |a|, \arg a)\in (0,\infty)\times \B{T}.
\]
The first component is the standard $\log$ for positive numbers. The second component is the argument of $a$ as a complex number.\footnote{In general, such a logarithmic map can be defined in each Galois embedding individually.} For simplicity, let us first look at the two integers $b_1,b_2.$ We claim that there are infinitely many positive rational integer pairs $(l_1,l_2)$ such that $b^{l_1}_1/b^{l_2}_2$ is close to one. In fact, consider the set \[P=\{b^{l_1}_1/b^{l_2}_2: (l_1,l_2)\in\mathbb{Z}^2_{>0}\}\subset\mathbb{C}.\] It is possible to see that $\overline{P}\cap S^1$ is not empty. However, it might be the case that $1\not\in \overline{P}.$ To see that this is not the case, observe that if $z\in \overline{P}\cap S^1,$ then $z^q\in \overline{P}\cap S^1$ for all $q\in\mathbb{Z}_{>0}.$ From here, and the fact that
\[
1\in \overline{\{z^q\}},
\]
we conclude the claim. 

Let $D_j\in D_{b_j}$ for $j\in \{1,\dots,k\}$ be digit sets with one missing digit in bases $b_1,\dots,b_k.$ This argument can be applied to $b_1,\dots,b_k$ and as a result, we see that there are large rational integers $l_1,l_2,\dots,l_k$ so that the convex fulls of 
\[
b^{l_1}_1C_{b_1,D_1},\dots,b^{l_k}_1C_{b_k,D_k}
\]
are almost the same. More precisely, recall that they are all close to some full squares. Those squares are basically the same square in the sense that the Hausdorff distance between them is much smaller than the diameters of those squares. This validates conditions (2,3) in Theorem \ref{thm: FY}.

We can now apply the argument in the proof of Theorem \ref{thm: main1}. As a result, we see that there are infinitely many complex numbers $z$ with the following property that
\[
z\in \cap_{j} A_{b_j,D_j}.
\]

Moreover, $|z|$ can be arbitrarily large. It is tempting to take the integer part of $z$ and claim that the `integer part' $[z]$ has digits in $D_j$ in base $b_j.$ However, due to our choice of the fundamental domain $C_{b_j,D_{b_j}}$, which is different for different $j$, the notion of taking the integer part is not the same for different $j$.\footnote{For $\B{K}=\mathbb{Q}$, for any $b>1$, the corresponding fundamental domain can be chosen to be the unit interval. There is no ambiguity in taking the integer part for any positive number.} For example, the information that $z\in A_{b_1,D_1}$ tells that 
\[
z=\sum_{l<0}r_lb^l_1+\sum_{l=0}^k r_l b^l_1
\]
with all $r_l\in D_1$. The integer part of $z$ with respect to the base $b_1$ is then
\[
z_1=\sum_{l=0}^k r_l b^l_1.
\]
Similarly, for $A_{b_2,D_2},$ we see that the integer part of $z$ with respect to the base $b_2$ is (for some integer $k'$ and digits $r_l\in D_2$)
\[
z_2=\sum_{l=0}^{k'} r'_l b^l_2.
\]
Thus we found two possibly different integers $z_1,z_2$. There is a finite set $B$ (depending on the field $\B{K}$) so that $z_1-z_2\in B$. Thus, although $z_1,z_2$ may be different, their difference can only have finitely many possible values (including zero). From here, we deduce the following result.
\begin{theorem}
    For each $k\geq 2,$ there is a number $M>0$ such that for each $k$ different integers in $\mathbb{Z}[i]$, say, $b_1,\dots,b_k$ with norm at least $M$, for each given choice of digit sets $D_i\subset D_{b_j}$ with one missing digit, there is a finite set $B$ such that there are infinitely many integers $z\in\mathbb{Z}[i]$ with the property that $z+s_j$ has digits in $D_j$ and $s_j\in B$ for all $j$. 
\end{theorem}

We now use a trick to rectify the fact that $s_i$ may not be always zero. In fact, if $b$ has a large norm, $D_b$ is the integer lattice $\mathbb{Z}^2$ contained in some large square. We can choose $D\subset D_b$ in a way that all `small' digits are removed from $D_b$. More precisely, for some $\epsilon\in (0,1)$, we remove all points in $D_b$ that are $\epsilon |b|$ close to the corners of the corresponding square. By choosing $\epsilon$ to be close to zero, we can achieve that $C_{b,D}$ has a large thickness. In fact, the thickness $C_{b,D}>\eta(b,\varepsilon)$ for some value $\eta(b,\varepsilon)$ which tends to $\infty$ as the norm of $b\to\infty$ and $\varepsilon\to 0$. Now, from the choice of $D$, we can guarantee that if $z$ has digits only in $D$, then $z+s$ does not have digit zero for all $s\in B.$  This implies that Theorem \ref{thm: FY} applies to systems of different such self-similar sets. From here, we finally proved the following result.

\begin{theorem}
    For each $k\geq 2,$ there is a number $M>0$ such that for each $k$ different integers in $\mathbb{Z}[i]$, say, $b_1,\dots,b_k$ with norm at least $M$, there are infinitely many integers $z\in\mathbb{Z}[i]$ with the property that $z$ misses the digit zero in base $b_j$ for all $j$. 
\end{theorem}

Now, we discuss the situation for general $\B{K}.$ We can extend all of the above constructions. However, for some $b\in O_{\B{K}}$, missing digits subsets of $C_{b,D_b}$ may not be self-similar sets. They are self-affine sets, so the thickness method no longer applies. Nonetheless, if $b$ has equal norm Galois embeddings (i.e. all Galois embeddings, real or complex, have the same norm\footnote{For example, rational integer multiples of roots of unity satisfy this property.}), then missing digits subsets of  $C_{b,D_b}$ are self-similar and the thickness method applies. We state the following result and omit its proof because it is similar to what we have proved with minor changes.

\begin{theorem}
     Let $\B{K}$ be a number field so that $O_{\B{K}}$ is an Euclidean domain. For each $k\geq 2,$ there is a number $M>0$ such that for each $k$ different integers in $O_{\B{K}}$, say, $b_1,\dots,b_k$ with equal norm Galois embeddings, and with norm at least $M$, there are infinitely many integers $z\in O_{\B{K}}$ with the property that $z$ misses the digit zero in base $b_j$ for all $j$. 
\end{theorem}

Finally, with only minor modifications in the above proof, the missing digit can be chosen to be any single digit individually in each base. The final version of the result in this section is given as follows.
\begin{theorem}
     Let $\B{K}$ be a number field so that $O_{\B{K}}$ is an Euclidean domain. For each $k\geq 2,$ there is a number $M>0$ such that for each $k$ different integers in $O_{\B{K}}$, say, $b_1,\dots,b_k$ with equal norm Galois embeddings, and with norm at least $M$, there are infinitely many integers $z\in O_{\B{K}}$ with the property that $z$ misses one arbitrarily fixed digit in base $b_j$ for all $j$. 
\end{theorem}
It is possible to generalise the above result, allowing more than one missing digit. The argument will be significantly more complicated because the topological thickness depends on the choice of digits in a dramatic manner. We thus stop obtaining further results.

\end{document}